\RequirePackage{fix-cm}
\documentclass[12pt,twoside,english,shortlabels]{amsart}       
\usepackage{graphicx}
\usepackage{mathptmx}      
\usepackage{latexsym}

\usepackage{amstext}
\usepackage{amssymb}

\usepackage{amsmath} 
\usepackage{latexsym}
\usepackage{graphicx} 

\usepackage{mathptmx}

\makeatletter

\usepackage{enumitem}

\theoremstyle{definition} 

 \newtheorem{definition}{Definition}[section]
 \newtheorem{remark}[definition]{Remark}

\newtheorem*{notation}{Notations}

\theoremstyle{plain}      

 \newtheorem{proposition}[definition]{Proposition}
 \newtheorem{theorem}[definition]{Theorem}

\newcommand{\cb}{\mathbb{C}}

\newcommand{\zb}{\mathbb{Z}}

\newcommand{\qb}{\mathbb{Q}}
\newcommand{\nb}{\mathbb{N}}

\newcommand{\lo}{{\rm Log\,}}

\def\11{{\mathbf 1}}

\newtheorem{exampl}[subsubsection]{Example}

\def\bee{\begin{exampl}}
\def\eee{\end{exampl}}
\def\bn{\begin{notation}}
\def\en{\end{notation}}
\def\br{\begin{remark}}
\def\er{\end{remark}}
\def\bp{\begin{prop}}
\def\ep{\end{prop}}
\def\bt{\begin{thm}}
\def\et{\end{thm}}
\def\be{\begin{equation}}
\def\ee{\end{equation}}
\def\bl{\begin{lem}}
\def\el{\end{lem}}
\def\bc{\begin{cor}}
\def\ec{\end{cor}}
\def\bd{\begin{defn}}
\def\ed{\end{defn}}

\title{ Alphabets, Rewriting Trails and Periodic Representations in Algebraic Bases
}

\numberwithin{equation}{subsection}

\author{Denys Dutykh$\dag$}
\thanks{}
\address{$\dag$Univ. Grenoble Alpes, Univ. Savoie Mont Blanc, CNRS UMR 5127, LAMA, 
F-73000 Chamb\'ery, France}
\email{Denys.Dutykh@univ-smb.fr}  

\author{Jean-Louis Verger-Gaugry$\ddag$}
\thanks{}
\address{$\ddag$Univ. Grenoble Alpes, Univ. Savoie Mont Blanc, CNRS UMR 5127, LAMA, 
F-73000 Chamb\'ery, France}
\email{Jean-Louis.Verger-Gaugry@univ-smb.fr}

\begin{document}



\maketitle

\begin{abstract}
For $\beta > 1$ a real algebraic integer 
({\it the base}),
the finite alphabets $\mathcal{A}
\subset \zb$ which realize
the identity $\qb(\beta) = 
{\rm Per}_{\mathcal{A}}(\beta)$,
where ${\rm Per}_{\mathcal{A}}(\beta)$ is the set of
complex numbers which are
$(\beta, \mathcal{A})$-eventually periodic
representations, are investigated.
Comparing with the greedy algorithm,
minimal and natural alphabets are defined.
The natural alphabets are shown to be 
correlated to the asymptotics of the
Pierce numbers of the base $\beta$ and Lehmer's problem.
The notion of rewriting trail is introduced
to construct intermediate alphabets
associated with small polynomial values of the base.
Consequences on the representations
of neighbourhoods of the origin
in $\qb(\beta)$,  
generalizing Schmidt's theorem related to Pisot numbers,
are investigated.
Applications to 
Galois conjugation
are given 
for convergent sequences of 
bases 
$\gamma_s := \gamma_{n, m_1 , \ldots , m_s}$  
such that
$\gamma_{s}^{-1}$ is the unique root in $(0,1)$
of an almost Newman polynomial of the type
$-1+x+x^n +x^{m_1}+\ldots+
x^{m_s}$, $n \geq 3$, 
$s \geq 1$, $m_1 - n \geq n-1$,
$m_{q+1}-m_q \geq n-1$ for all
$q \geq 1$.
For $\beta > 1$ a
reciprocal algebraic integer close to one,
the poles of modulus $< 1$
of the dynamical zeta function of
the $\beta$-shift $\zeta_{\beta}(z)$
are shown, under some assumptions,
to be zeroes of the minimal polynomial of
$\beta$.
\end{abstract}


\vspace{0.4cm}



\tableofcontents

\section{Introduction}
\label{S1}

For a general complex number $\beta \in \cb$,
$|\beta| > 1$, and a finite alphabet
$\mathcal{A} \subset \cb$,
we define the
$(\beta, \mathcal{A})$-representations as
expressions of the form
$\sum_{k \geq -L} a_k \beta^{-k}$,
$a_k \in \mathcal{A}$,
for some integer $L \in \zb$.
They are Laurent series of $1/\beta$.
We define
$${\rm Per}_{\mathcal{A}}(\beta)
:=
\{ x \in \cb \mid x\, {\rm \,has \,an \,eventually 
\,periodic}
\, (\beta, \mathcal{A}){\rm -representation}\}.
$$
In this note attention is focused on
the complex numbers $\beta$ which are real algebraic
integers $> 1$,
close to 1, assuming that
$\beta$ has no conjugate on the unit circle,
and on the alphabets 
$\mathcal{A} \subset \zb$, depending upon $\beta$,
involved
in the identity:
$$\qb(\beta) = {\rm Per}_{\mathcal{A}}(\beta).$$
We write $\qb$ for the set of rational numbers, 
$\qb(\beta)$ for the smallest sub‐field of $\cb$ 
containing $\beta$.
Indeed, such an identity always holds by the 
following theorem.

\begin{theorem}[Kala -V\'avra \cite{kalavavra}]
\label{kalavavra}
Let $\beta \in \cb$ be an algebraic number 
of degree $d$,
$|\beta| > 1$,
and let $a_d x^d - a_{d-1} x^{d-1} - \ldots
- a_1 x - a_0 \in \zb[x]$ be its minimal polynomial. 
Suppose that
$|\beta'| \neq 1$ for any conjugate 
$\beta'$ of $\beta$.
Then there exists a finite alphabet
$\mathcal{A} \subset \zb$ such that
\begin{equation}
\label{QbetaPer}
\qb(\beta) = {\rm Per}_{\mathcal{A}}(\beta).
\end{equation}
\end{theorem}
Theorem \ref{kalavavra}
is a generalization
of a previous theorem of 
Baker, Mas\'akov\'a, Pelantov\'a and V\'avra
\cite{bakermasakovapelantovavavra}
in which $1/a_d$ was assumed to belong to
$\zb[\beta, \beta^{-1}]$,
an assumption removed in
\cite{kalavavra}.

In Section \ref{S2} we revisit the
construction of an alphabet
$\mathcal{A} \subset \zb$,
symmetrical 
with respect to the origin,
 which
allows \eqref{QbetaPer} to hold, given in
\cite{frougnypelantovasvobodova}. 
We show that the size of this alphabet
is correlated to the Pierce numbers
$\Delta_{N}(\beta)$
of $\beta$. The 
numerical explosion 
of $\Delta_{N}(\beta)$ with $N$
has been 
investigated in \cite{einsiedlereverestward}.
Pierce numbers play an important role
in the 
Mahler measure of $\beta$
and the search of big prime numbers
(Lehmer
\cite{lehmer}, Einsiedler, Everest and Ward
\cite{einsiedlereverestward}).
The alphabet constructed by this means 
is called the {\it natural 
alphabet realizing}  $\eqref{QbetaPer}$.
We denote it by $\mathcal{A}^{(nat)}_{\beta}$.
It has no reason to be the smallest one
realizing  $\eqref{QbetaPer}$.

\begin{remark}
Denote by 
$$\widehat{\mathcal{A}} := \Bigl\{\{-m, -m+1, \ldots,
-1, 0, +1,
\ldots, m-1, m\} \mid m \in \nb\setminus \{0\}
\Bigr\}$$
the set of symmetrical alphabets with digits in
$\zb$. 
It is totally ordered by inclusion.
If $\mathcal{A}_1 =
\{-m_1, \ldots, 0,\ldots, m_1\}$, $\mathcal{A}_2
=
\{-m_2, \ldots, 0,\ldots, m_2\}$, are two elements
of $\widehat{\mathcal{A}}$, then
$$\mathcal{A}_1 \subset \mathcal{A}_2
\qquad {if ~and~ only~ if}
\qquad
m_1 \leq m_2.$$
The explicit construction
of the map $\beta \to \mathcal{A}_{\beta}^{(nat)}
\in \widehat{\mathcal{A}}$,
as in Section \ref{S2}, proves the existence
of at least one alphabet
say $\mathcal{A}_{\beta}
 \in \widehat{\mathcal{A}}$ realizing 
 $\eqref{QbetaPer}$, included 
 (a priori not necessarily strictly) in
 $\mathcal{A}_{\beta}^{(nat)}$.
This justifies the terminology ``natural"
for $\mathcal{A}_{\beta}^{(nat)}$.
Let us note that,
if a
finite alphabet $\mathcal{A}_{\beta}
\in \widehat{\mathcal{A}}$ realizes (1), 
then any of its finite
supersets does that, and could be bigger than
$\mathcal{A}_{\beta}^{(nat)}$. 
Therefore there
is interest in characterizing 
the symmetrical alphabets 
$\mathcal{A}_{\beta} 
\subset \mathcal{A}^{(nat)}_{\beta}$
which 
realize  $\eqref{QbetaPer}$. 
Because of the total ordering of
$\widehat{\mathcal{A}}$, among all of them,
there is an unique smallest element, say
$\mathcal{A}_{mini}$.
We have
\vspace{0.1cm}

$$\{-1, 0, +1\} 
\subset
\ldots 
\subset 
\mathcal{A}_{mini}
\subset
\ldots 
\subset
\mathcal{A}_{\beta}
\subset
\ldots 
\subset
\mathcal{A}^{(nat)}_{\beta}
\subset
\ldots 
$$
\end{remark}

\

\noindent
{\bf Problem}: For $\beta$ any 
real algebraic
integer $> 1$ such that
$\beta$ has no conjugate on the unit circle,
what is the minimal symmetrical
alphabet 
$\mathcal{A} \subset \zb$,
$\mathcal{A} \in \widehat{\mathcal{A}}$,
realizing 
$\eqref{QbetaPer}$?

\

The minimal alphabet 
$\mathcal{A}_{mini}$ a priori depends upon $\beta$.
Intermediate alphabets between
$\mathcal{A}_{mini}$ and
$\mathcal{A}^{(nat)}_{\beta}$ realizing  
$\eqref{QbetaPer}$ are 
investigated by introducing
rewriting trails in Section \ref{S3}.

If $\beta$ is a Pisot number the problem is solved
by
the following theorem \cite{schmidt},
with the {\it minimal alphabet}
$\mathcal{A} = \{-1, 0, +1\}$
(independent of $\beta$).  
The set
${\rm Per}_{\{0,1\}}(\beta)$ is
the set of (eventually) periodic points for 
the $\beta$-transformation 
$T_{\beta}: x \to \{\beta x\}$ on
$[0,1)$, i.e. for the set of points 
whose orbits under $T_{\beta}$, are finite. 
The $(\beta,\{-1,0,1\} )$-eventually
periodic representations of the elements
$x \in \qb(\beta) \cap (-1, +1)$ are 
the R\'enyi expansions,
equivalently they are constructed from 
the greedy algorithm. 
Then all the elements of $\qb(\beta)$ 
have eventually
periodic representations.

\begin{theorem}[K. Schmidt \cite{schmidt}]
\label{schmidttheorem}
Let $\beta > 1$ be a real number.

(1) If $\qb \cap [0,1) \subset{\rm Per}_{\{0,1\}}(\beta)$, then $\beta$ is either a Pisot or a Salem number.

(2) If $\beta$ is a Pisot number, then 
${\rm Per}_{\{0,1\}}(\beta) = \qb(\beta) \cap [0,1]$.
\end{theorem}
If $\beta$ is a Pisot number 
and
$x \in [-1,0] \cap \qb(\beta)$, then 
$-x$ admits an eventually periodic
representation in base $\beta$, with digits
in $\{-1, 0\}$, which is the opposite of the one
of $|x|$, so that
any $x \in [-1,+1]$, hence any $x \in \qb(\beta)$,
has an eventually periodic 
representation with digits in 
the symmetric alphabet
$\{-1,0,+1\}$.
By comparison,
the natural alphabets 
$\mathcal{A}^{(nat)}_{\beta_k}$
associated to the Pisot numbers $\beta_k$
belonging to an increasing sequence  
tending to $(1+\sqrt{5})/2$,
calculated by means of Proposition \ref{poly_dominant},
are studied 
in Section \ref{S2.2}.

Dar\'oczy and K\'atai \cite{daroczykatai},
and later Thurston \cite{thurston},
have proved that for any non-real $\beta \in \cb$,
$|\beta| > 1$, there exists a finite alphabet
$\mathcal{A} \subset \cb$ such that every 
$x \in \cb$ has a 
$(\beta, \mathcal{A})$-representation.
The search for periodic representations
in radix systems goes back
to Kov\'acs \cite{kovacs} and
to Kov\'acs and K\"ornyei \cite{kovacskornyei}
(see also Peth\H{o}
\cite{pethoe}). 
For the R\'enyi-Parry numeration
system in base $\beta > 1$, 
the idea of the enlargement of the alphabet
to obtain the eventual periodicity 
for the representations of the 
elements of the number field
$\qb(\beta)$ is recurrent.

Theorem \ref{mainthm} extends 
Schmidt's Theorem 
\ref{schmidttheorem}
to the 
representations 
of the elements 
of $\qb(\gamma) \cap \mathcal{V}$ 
where $\mathcal{V}$ is a neighbourhood of
the origin, 
and $\gamma > 1$ an algebraic integer,
root of a polynomial
with coefficients in $\{-1,0,1\}$, 
having no  conjugate on the unit circle.
In Section
\ref{S3.1} we introduce the notion 
of {\it rewriting trail}. 
We show that
intermediate alphabets, between the minimal 
and the natural ones,
are produced by rewriting trails.
The proof of Theorem \ref{mainthm}
is based on rewriting trails, and
makes use of Kala - Vavra's Theorem \ref{kalavavra}.

\begin{theorem}
\label{mainthm}
Let $\gamma > 1$ be an algebraic integer, 
root
of a polynomial 
$S_{\gamma}(X) = X^s -
\sum_{i=0}^{s-1} t_{s-i} X^i$, with $s \geq 1,
t_i \in \zb, |t_i| \leq 1$, 
not necessarily irreducible, such that
$|\gamma'| \neq 1$ for any conjugate 
$\gamma'$ of $\gamma$.

Let $P(X)=1 + a_1 X +a_2 X^2 + \ldots + 
a_{d-1} X^{d-1} +a_d X^d \in \zb[X]$,
$d =\deg P \geq 1$, be an integer polynomial. 
Denote by $H=\max_{i=1,\ldots,d} |a_i|$
the height of $P$. 

Let $0 < \eta < 1$ and suppose 
$0 \neq |P(\gamma)|< \eta$. Then
the polynomial value 
$P(\gamma) \in \qb(\gamma)$ 
admits at least one eventually periodic
representation
\begin{equation}
\label{pervaluesHd}
P(\gamma) 
=
R(\gamma^{-1}) + \frac{1}{\gamma^L} 
\sum_{j=0}^{\infty} \frac{1}{\gamma^{j r }} 
T(\gamma^{-1})
\qquad \in {\rm Per}_{\mathcal{A}}(\gamma)
\end{equation}
with \begin{enumerate}
\item[(i)]
alphabet 
$\mathcal{A} = \{-m, \ldots, +m\} \subset \zb$\,,
$m=\lceil2((2^{d} -1) H + 2^d)/3\rceil$,
independent of $s$ and $\gamma$,
\item[(ii)] $R(X) \in \mathcal{A}[X]$, 
$\deg R \leq s-1$,
$T(X) \in \mathcal{A}[X]$,
$\deg T \leq s-1$,
and 
$L$ and $r$ being some integers satisfying
$L > \deg R$,
$r > \deg T$,
\item[(iii)] preperiod
$$R(\gamma^{-1}) = \frac{a_w}{\gamma^w}
+\frac{a_{w+1}}{\gamma^{w+1}}
+\ldots
+\frac{a_{w+s-1}}{\gamma^{w+s-1}},
\qquad a_j \in \mathcal{A}, 
j=w,  \ldots, w+s-1,~
a_w \neq 0,$$
with $w \geq 1$ satisfying 
$\frac{\kappa_{\gamma, \mathcal{A}}}{\eta} 
\leq \gamma^{w-1}$ for some positive constant
$\kappa_{\gamma, \mathcal{A}}$ depending upon
$\gamma$ and $\mathcal{A}$.
\end{enumerate}

\end{theorem}

\begin{remark}
In Theorem \ref{mainthm}
the polynomial $S_{\gamma}(X)$ could have some zeroes
of modulus one. For instance, if
it is of the form $S_{\gamma}(X)=A(X) \times C(X)$
with $A(X)$ a product of cyclotomic polynomials
and $C(X)$ the minimal polynomial of
$\gamma$. The assumption that the conjugates
$\gamma'$ do not lie on the unit circle only
concerns the zeroes of $C(X)$.
\end{remark}

In Section \ref{S3.2} Theorem \ref{mainthm} is applied
to the Galois conjugation of eventually periodic
representations of 
polynomial values
of the base
$\gamma$ for $\gamma$ runing over a sequence
of real algebraic integers converging towards
a reciprocal algebraic integer $\beta > 1$. 
The consequences 
on the 
Galois conjugates of $\beta$ of modulus $< 1$ 
are investigated 
in the  
context
of 
automorphisms of complex numbers
(Kestelman \cite{kestelman},
Yales \cite{yales}); 
the absence of continuity of the 
$\qb$-automorphisms of conjugation
is compensated in some sense by
the eventual periodicity of the representations.
Proposition \ref{omegazerofbetaPbeta} 
reports some 
consequences on the relations
between the poles of the dynamical
zeta function $\zeta_{\beta}(z)$ of the
$\beta$-shift (see e.g. Solomyak
\cite{solomyak}) and the zeroes of the
minimal polynomial of $\beta$.
Examples of natural alphabets related to
sequences of polynomials of the 
class $\mathcal{B}$
are studied in
Section \ref{S3.3}, in terms of sequences
of Mahler measures.

\section{Natural alphabets in $(\beta, \mathcal{A})$-periodic representations of $\qb(\beta)$}
\label{S2}

Let $t \geq 1$. 
A polynomial $Q(X) =\sum_{i=0}^{d} a_i X^i
\in \zb[X]$ 
is said to have a {\it dominant coefficient},
resp. to be a {\it $t$-polynomial},
if there exists an integer  
$j \in \{0, 1, \ldots, d\}$ such that
$|a_j| > \sum_{i=0, i \neq j}^d |a_i|$,
resp.
$|a_j| > t \,\sum_{i=0, i \neq j}^d |a_i|$ .
Let $\beta$ be
an algebraic integer $> 1$ having no
conjugate on the unit circle.
If the ideal $(P_{\beta})$
$= P_{\beta}(X) \zb[X]$
generated by
the minimal polynomial 
$P_{\beta}(X)$
of $\beta$ contains a
$1$-polynomial
$\sum_{i=0}^{d} a_i X^i$
, of dominant coefficient
$a_j$, then, by Proposition 5.1 in
\cite{frougnypelantovasvobodova}
and Theorem 25 in
\cite{bakermasakovapelantovavavra}, 
the alphabet
\begin{equation}
\label{Alphabet}
\{-m , \ldots , 0,\ldots, m\}, \qquad
{\rm with} \qquad 
m:=\lceil\frac{|a_{j}|-1}{2}\rceil
+
\sum_{i=0, i \neq j}^{d} |a_i|,
\end{equation}
satisfies \eqref{QbetaPer}.
Here $\lceil \,\rceil$ denotes the upper integer part.
In Section \ref{S2.1} we recall
an effective construction of such a 
$1$-polynomial
in $(P_{\beta})$. The proof
of Proposition
\ref{poly_dominant} is reproduced from
\cite{frougnypelantovasvobodova} to fix the notations.
The way it is obtained
comes from a necessarily finite number of 
successive iterations of the companion
matrix of $P_{\beta}$.

\subsection{Pierce numbers of the base and integer polynomials with a dominant coefficient}
\label{S2.1}

\begin{proposition}
\label{poly_dominant}
Let $\alpha$ be an algebraic integer, of
degree $d$, $|\alpha|>1$,
of minimal polynomial
$P_{\alpha}(X) = \prod_{j=1}^{d} (X - \alpha_{(j)})$,
with $\alpha = \alpha_{(1)}$ and
$|\alpha_{(j)}|\neq 1$ for 
$j=2, 3, \ldots, d$. Denote
by $j_0$ the number of conjugates
$\alpha_{(j)}$ of $\alpha$ which have
a modulus $> 1$.
Then, for any $t \geq 1$, there exist
an integer $N$ and a polynomial
$$Q(X) = 
X^{d N} +
a_{1} X^{(d-1) N}
+ a_{2} X^{(d-2) N}
+\ldots
+a_{d-1} X^{N}
+ a_d \quad \in \zb[X]  
$$
such that
$Q(\alpha)= 0$, setting
$a_0 =1$, with
\begin{equation}
\label{dominantCOEFFineq}
|a_{j_0}|
> ~~
t \sum_{i \in \{0, 1, 2, \ldots, d\}
\setminus \{j_0\}} |a_i| .
\end{equation}
\end{proposition}

\begin{proof}
We have $j_0 \geq 1$. 
The minimal
polynomial 
$$P_{\alpha}(X) = \prod_{j=1}^{d} (X - \alpha_{(j)})
=
X^d + g_1 X^{d-1}
+ g_2 X^{d-2}
+\ldots
+
g_{d-1} X + g_d \in \zb[X]$$ 
can be written as the characteristic polynomial
of $\alpha$,
from the 
companion matrix \cite{lancaster}
$$
H = \left(
\begin{array}{cccccc}
0&0& &\ldots&0&-g_d\\
1&0&0&\ldots&0&-g_{d-1}\\
0&1&0&\ldots&0&-g_{d-2}\\
\vdots&&&&&\\
\vdots &&&1&0&-g_{2}\\
0&\ldots&0&&1&-g_{1}
\end{array}
\right).
$$
We have: $\det(H - X \,{\rm I}_d) = 
(-1)^d P_{\alpha}(X)$,
where I$_d$ is the identity matrix.
The eigenvalues of $H$ are
the zeroes of $P_{\alpha}(X)$.
For $n \geq 2$ let us define
$$P_{\alpha, n}(X) :=
(-1)^d
\det(H^n - X \,{\rm I}_d) \in \zb[X].$$
The polynomial $P_{\alpha, n}(X)$,
of degree $d$, has integer coefficients
$$P_{\alpha, n}(X) = 
\prod_{j=1}^{d} (X - \alpha_{(j)}^{n})
=
X^d + g_{1}(n) X^{d-1}
+ g_{2}(n) X^{d-2}
+\ldots
+
g_{d-1}(n) X + g_{d}(n).$$ 
We set: $g_j = g_{j}(1)$ for $j=1, 2, \ldots, n$
and $g_0 = g_{0}(n) =1$
for $n \geq 1$.
The coefficients $g_{j}(n)$ are related to the symmetric
functions of the roots.
Without loss of generality, let us assume:
$$|\alpha_{(1)}| \geq
|\alpha_{(2)}| \geq
\ldots
\geq
|\alpha_{(j_0)}|
>
|\alpha_{(j_0 + 1)}|
\geq
\ldots
\geq
|\alpha_{(d)}|,
$$
where $j_0:=\max\{i : 1 < |\alpha_{(i)}|\}$.
The choice of $j_0$ guarantees
$$\left|
\frac{\alpha_{(i_1)} \alpha_{(i_2)} \ldots 
\alpha_{(i_r)}}{\alpha_{(1)} \alpha_{(2)} \ldots 
\alpha_{(j_0)}}
\right| < 1$$
for any subset
$\{i_1 , i_2 , \ldots , i_r \} \subset 
\{1, 2, \ldots , d\}$
and
$\{i_1 , i_2 , \ldots , i_r \} \neq 
\{1, 2, \ldots , j_0\}$.
Then, for all choices
of
$\{i_1 , i_2 , \ldots , i_r \} \neq 
\{1, 2, \ldots , j_0\}$, we have:
$$\lim_{n \to \infty}\,
\frac{\alpha_{(i_1)}^n \alpha_{(i_2)}^n \ldots 
\alpha_{(i_r)}^n}{\alpha_{(1)}^n \alpha_{(2)}^n \ldots 
\alpha_{(j_0)}^n} = 0.$$
Now, for all $n \geq 1, 1 \leq j \leq d$, we have:
$$
g_{j}(n)
=
\sum_{\{i_1 , i_2 , \ldots, i_j\} \in S_j}
\alpha_{(i_1)}^n \alpha_{(i_2)}^n \ldots 
\alpha_{(i_j)}^n
$$
where $S_j =
\{\mathcal{P} \subset \{1, 2, \ldots, d\} :
\# \mathcal{P} = j\}$ 
is the set of all subsets
of 
$\{1, 2, \ldots, d\}$ with cardinality $j$.
Since
$$\lim_{n \to \infty}\, 
\frac{g_{j}(n)}
{\alpha_{(1)}^n \alpha_{(2)}^n \ldots 
\alpha_{(j_0)}^n} = \left\{
\begin{array}{cc}
0& {\rm for ~all}~ 
j = 1, 2, \ldots, d {\rm ~~and~} j \neq j_0\\
(-1)^j& {\rm for~} j = j_0 
\end{array}
\right. ,$$
we deduce that, for any $t > 0$, there exists
an integer $N = N(t)$ such that
$$
\frac{|g_{j_0}(N)|}
{|\alpha_{(1)}^N \alpha_{(2)}^N \ldots 
\alpha_{(j_0)}^N |}~~
> ~~~~t \!\!\!\!\sum_{j \in \{0, 1, 2, \ldots , d\}, j \neq j_0}
\frac{|g_{j}(N)|}
{|\alpha_{(1)}^N \alpha_{(2)}^N \ldots 
\alpha_{(j_0)}^N |}
$$
equivalently
\begin{equation}
\label{001condition}
|g_{j_0}(N)|
> ~~t \sum_{j \in \{0, 1, 2, \ldots , d\}, j \neq j_0}
|g_{j}(N)|.
\end{equation}
This inequality gives the result 
\eqref{dominantCOEFFineq},
with 
$Q(X) = P_{\alpha, N}(X^N)$. 
\end{proof}
\begin{definition}
The smallest integer $N = N(t)$ for which 
\eqref{001condition}
is satisfied is called the 
{\it dominance index of $P_{\alpha}$
(or of $\alpha$)
for the value $t \geq 1$}. For $t=1$,
$N(1)$ is called the {\it dominance index of
$P_{\alpha}$} (or of $\alpha$).
\end{definition}

\begin{definition}
Let $\alpha > 1$ be a real algebraic integer.
With the same notations as in 
Proposition \ref{poly_dominant} and its proof,
the alphabet 
$:=
\{-m , \ldots , 0,\ldots, m\}$, 
with
$$m :=\lceil\frac{|g_{j_0}(N)|-1}{2}\rceil
+
\sum_{j=0, j \neq j_0}^{d} |g_{j}(N)|
$$
and $N$ the dominance index of
$\alpha$, is called
the {\it natural alphabet} of $\alpha$,
and denoted by
$\mathcal{A}^{(nat)}_{\alpha}$.
\end{definition}

For $\alpha > 1$ any real algebraic integer,
the existence of the natural alphabet
$\mathcal{A}^{(nat)}_{\alpha}$ implies
that $\alpha$ satisfies the 
{\it weak representation of zero property},
or, for short, $\alpha$ is {\it WRZ},
in the terminology of 
\cite{frougnypelantovasvobodova}.
Then, in the Sections 4 and 5 in
\cite{frougnypelantovasvobodova},
Frougny, Pelantova and Svobodova
provide a parallel algorithm
``Algorithm II" 
which
gives access to  \eqref{QbetaPer}.

\begin{proposition}
\label{pierceminorant}
Let $\alpha > 1$ be a real algebraic integer.
With the same notations as in 
Proposition \ref{poly_dominant} and its proof,
with $N=N(t)$ the smallest value
which satisfies \eqref{001condition},
we have:
\begin{equation}
\label{002condition}
|g_{j_0}(N)|
> \frac{t}{1+t} \Delta_{N}(\alpha),
\end{equation}
where
$|P_{\alpha,N}(1)| = \Delta_{N}(\alpha)
=
\bigl|
\prod_{j=1}^{d} (1 - \alpha_{(j)}^{N})
\bigr|
$ 
is the $N$-th Pierce number
of $\alpha$,
and
the natural alphabet 
$\mathcal{A}^{(nat)}_{\alpha}
=
\{-m , \ldots , 0,\ldots, m\}$
is such that
$$m 
\geq
\lceil 2^{-1} (2^{-1} \Delta_{N}(\alpha)-1)\rceil .
$$
\end{proposition}

\begin{proof}
In the continuation of \eqref{001condition}, 
we have
$$\sum_{j \in \{0, 1, 2, \ldots , d\}, j \neq j_0}
|g_{j}(N)| \geq 
\Bigl|\sum_{j=0}^{d} g_{j}(N) - g_{j_0}(N)\Bigr|
=\Bigl|P_{\alpha,N}(1) - g_{j_0}(N)\Bigr|$$
$$\geq \Bigl||P_{\alpha,N}(1)| - |g_{j_0}(N)|\Bigr|
\geq
\Delta_{N}(\alpha) - |g_{j_0}(N)| .
$$
Therefore
\begin{equation}
\label{001conditionPierce}
|g_{j_0}(N)|
> ~~t \,
 \Delta_{N}(\alpha) -  t \,|g_{j_0}(N)|,
\end{equation}
equivalently
\eqref{002condition}.
We now take $t=1$, $N=N(1)$ the dominance index of $\alpha$,
and apply Proposition 5.1 in
\cite{frougnypelantovasvobodova}
and Theorem 25 in
\cite{bakermasakovapelantovavavra}.
\end{proof}

\begin{remark}
\label{remarkSysDyn02}
To each polynomial of the form 
$P_{\alpha}$ as in Proposition \ref{poly_dominant}
 there is an associated endomorphism 
 ${\rm T}_{P_{\alpha}}$ of the
$d$-torus, given by the natural action of the companion matrix of $P_{\alpha}$. 
${\rm T}_{P_{\alpha}}$
is an ergodic transformation 
with respect to Lebesgue measure, and
$\Delta_{N}(P_{\alpha})$ is the number 
of points of period N under ${\rm T}_{P_{\alpha}}$
\cite{everestward}. 
The Mahler measure 
${\rm M}(\alpha)=
\prod_{i=1}^{d}
\max\{1, |\alpha_{(i)}|\}$
of $\alpha$ is related to
the dynamical properties of the 
corresponding toral endomorphism.
The condition of having no root
on the unit circle implies
expansiveness of ${\rm T}_{P_{\alpha}}$ 
as a topological
dynamical system.
The topological entropy of 
${\rm T}_{P_{\alpha}}$
is equal to $\lo {\rm M}(\alpha)$
\cite{lind}. 
\end{remark}

\begin{remark}
The link between the natural alphabet
$\mathcal{A}^{(nat)}_{\beta}$
and the Mahler measure
${\rm M}(\alpha)$
of the base of numeration $\alpha$
naturally comes from Proposition
\ref{poly_dominant} where $j_0$
counts the number of roots 
outside the closed unit disk.
It can be estimated roughly as follows:
first the $N$-th Pierce number
of $\alpha$ is
$$
\Delta_{N}(\alpha) =
\frac{\Delta_{N}(\alpha)}{\Delta_{N-1}(\alpha)}
\times
\frac{\Delta_{N-1}(\alpha)}{\Delta_{N-2}(\alpha)}
\times
\ldots
\times
\frac{\Delta_{2}(\alpha)}{\Delta_{1}(\alpha)}
\Delta_{1}(\alpha),
$$
with $\Delta_{1}(\alpha) = |P_{\alpha}(1)|$.
From Lehmer \cite{lehmer}, 
 $${\rm M}(\alpha) =
\lim_{q \to \infty} \frac{\Delta_{q+1}(\alpha)}
{\Delta_{q}(\alpha)},$$
we deduce,
without taking into account the type of convergence
towards ${\rm M}(\alpha)$, as a rough estimate 
for the lower bound,
\begin{equation}
\label{002conditionMahler}
|g_{j_0}(N)|
> \frac{1}{2} {\rm M}(\alpha)^{N-1}
|P_{\alpha}(1)|,
\end{equation}
and the approximate lower bound 
$\lceil 2^{-1} (2^{-1}{\rm M}(\alpha)^{N-1}
|P_{\alpha}(1)| -1)\rceil$
for $m$. Let us note that
the sequence $(\Delta_{n}(\alpha))_n$ is fairly chaotic,
as the sequence
of the Pierce numbers of $\alpha$,
from the heuristics of 
Einsiedler, Everest and Ward
\cite{einsiedlereverestward}.
\end{remark}

\subsection{Natural alphabets for a convergent sequence of Pisot numbers}
\label{S2.2}

In this paragraph we examplify the
hugeness of the natural alphabets for a sequence
of Pisot numbers.

Let us consider the sequence
of irreducible integer polynomials (from Theorem 7.2.1
in \cite{bertinetal})
$$P_{2 k}(z)
=
(1 - z^{2 k} (1+z-z^2))/(1-z), \qquad k \geq 1.$$
The dominant root of $P_{2 k}(z)$ is denoted
$\beta_{k} > 1$. All the other roots
have a modulus $< 1$.
For all $k \geq 1$, 
we have: 
$\beta_{k} < (1+\sqrt{5})/2$.
The sequence 
$(\beta_{k})_{k \geq 1}$ is an increasing
sequence of Pisot numbers, with limit:
$\lim_{k \to \infty}
\beta_{k} = \frac{1+\sqrt{5}}{2}$.
For $k=1$ we have
$P_{2}(z) = z^3 - z -1$ and $\beta_1$ is 
the smallest Pisot number.
Let $\tau = \beta_{\infty} 
= (1+\sqrt{5})/2$.
It is the dominant root of the trinomial
$z^2 - z -1$. 

The dominance index of
$\tau$ is 3, and the
natural alphabet
$\mathcal{A}^{(nat)}_{\tau}$
is 
$=
\{-3, \ldots, +3\}
$.
The growth rate of the natural alphabet
$\mathcal{A}^{(nat)}_{\beta_k}
=
\{-m_{k} , \ldots, m_{k}\}
$ is represented as a function of $k$
in Figure \ref{growthpisotMAXalphabets}.
\begin{figure}
\begin{center}
\includegraphics[width=9cm]{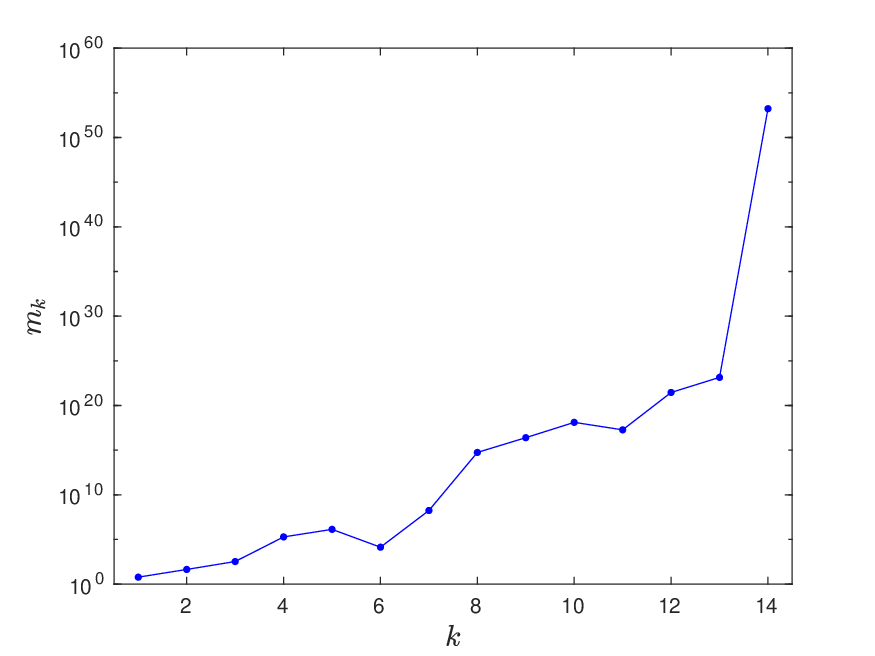}
\end{center}
\caption{Natural alphabets of the Pisot numbers
$\beta_k$.}
\label{growthpisotMAXalphabets}
\end{figure}

\section{Small heights and eventually periodic representations of polynomial values of the base}
\label{S3}

\subsection{Rewriting trails, intermediate alphabets - Proof of Theorem \ref{mainthm}}
\label{S3.1}

Denote by
$S_{\gamma}^{*}(X)=X^s S_{\gamma}(1/X)
=
1 -t_{1} X -t_2 X^2 - \ldots - t_{s-1} X^{s-1} 
-t_s X^s $
the reciprocal polynomial of
$S_{\gamma}(X)= X^s -
\sum_{i=0}^{s-1} t_{s-i} X^i$. The coefficients
$t_i$ are in $\{-1,0,+1\}$.
The algebraic integer $\gamma$ is called the
{\it base} and $S_{\gamma}^{*}(\gamma^{-1})=0$.

We want to express
$P(\gamma)$ 
as a $(\gamma, \mathcal{A})$- 
eventually
periodic representation
with a certain alphabet $\mathcal{A}$
to be defined.
This objective means that, first,
we have to express $P(\gamma)$
as a Laurent series of $1/\gamma$.

We now introduce a construction,
that we call ``rewriting trail
from ``$S_{\gamma}^{*}$"
to ``$P$", at $\gamma^{-1}$", 
to reach this objective, and which will
allow us to
show that a symmetrical alphabet 
$\mathcal{A} =\{-m, \ldots, 0, \ldots, m\}$
can be defined and
is such that $m$
depends upon
$H$ and $\deg(P)$, independently of 
$s$ and $\gamma$.

The starting point is the identity
$1 = 1$, to which we add 
$0= - S_{\gamma}^{*}(\gamma^{-1})$
in the right hand side.
Then we define a rewriting trail from
\begin{equation}
\label{Sgamma}
1=1-S_{\gamma}^{*}(\gamma^{-1})
=
t_{1}\gamma^{-1}  +t_2 \gamma^{-2} + \ldots 
+ t_{s-1} \gamma^{-(s-1)} +t_s \gamma^{-s}
\end{equation}
to
$$
- a_1 \gamma^{-1} - a_2 \gamma^{-2} + \ldots - 
a_{d-1} \gamma^{-(d-1)} - a_d \gamma^{-d}
= 1 - P(\gamma^{-1}).
$$
A rewriting trail will be a
sequence of integer polynomials, whose
role will
consist in ``restoring" the coefficients
of $1 - P(\gamma^{-1})$ one after the other, 
from the left,
by adding ``$0$" conveniently at each step
to both sides of \eqref{Sgamma}.
At the first step we add $0=
(-a_1 - t_1) \gamma^{-1} S_{\gamma}^{*}(\gamma^{-1})$;
and we obtain 
$$1= -a_1 \gamma^{-1}$$
$$+(-(-a_1 -t_1) t_1 + t_2) \gamma^{-2}
+(-(-a_1 -t_1) t_2 + t_3) \gamma^{-3} + \ldots
$$
so that the height of the polynomial
$$(-(-a_1 -t_1) t_1 + t_2) X^{2}
+(-(-a_1 -t_1) t_2 + t_3) X^{3} + \ldots
$$
is $\leq H+2$.
At the second step we add
$0=
(-a_2 - (-(-a_1 -t_1) t_1 + t_2)) \gamma^{-2} S_{\gamma}^{*}(\gamma^{-1})
$.
Then we obtain
$$1= -a_1 \gamma^{-1} - a_2 \gamma^{-2}$$
$$+
[(-a_2 - (-(-a_1 -t_1) t_1 + t_2))t_1
+ (-(-a_1 -t_1) t_2 + t_3)] \gamma^{-3}+\ldots
$$
where the height of the polynomial
$$[(-a_2 - (-(-a_1 -t_1) t_1 + t_2))t_1
+ (-(-a_1 -t_1) t_2 + t_3)] X^{3}+\ldots
$$
is $\leq H + (H+2)+(H+2)=3 H+4$.
Iterating this process $d$ times 
we obtain
$$1= -a_1 \gamma^{-1} - a_2 \gamma^{-2} -\ldots
- a_d \gamma^{-d}$$
$$+~~
polynomial ~~remainder~~ in~~ \gamma^{-1}.
$$
Denote by $V(\gamma^{-1})$
this polynomial remainder in $\gamma^{-1}$,
for some $V(X) \in \zb[X]$,
and $X$ specializing in $\gamma^{-1}$.
If we denote the upper bound of the
height of the polynomial remainder
$V(X)$, 
at step $q$, by $\lambda_q H + v_q$, 
we readily
deduce: $v_q = 2^q$, and
$\lambda_{q+1} = 2 \lambda_{q} +1$, $q \geq 1$,
with $\lambda_1 = 1$; then 
$\lambda_{q} = 2^{q}-1$.

To summarize,
we obtain a sequence
$(A'_{q}(X))_{q \geq 1}$ of rewriting polynomials
involved in this rewriting trail;
for $q \geq 1$, $A'_{q} \in \zb[X]$,
$\deg(A'_{q}) \leq q$
and $A'_{q}(0)= -1$. 
The first polynomial
$A'_{1}(X)$ is 
$-1 + (-a_1 - t_1) X$. The second 
polynomial
$A'_{2}(X)$ is 
$-1 + (-a_1 - t_1) X +
(-a_2 - (-(-a_1 -t_1) t_1 + t_2)) X^2$,
etc.

For $q \geq \deg(P)$, all the coefficients 
of $P$ are restored; denote
by
$(h'_{q,j})_{j=0,1,\ldots,s-1}$ the $s$-tuple
of integers produced by this rewriting trail,
at step $q$. It is such that
\begin{equation}
\label{AprimeSPreste}
A'_{q}(\gamma^{-1}) S_{\gamma}^{*}(\gamma^{-1})
=
-P(\gamma^{-1}) + \gamma^{-q-1}
\Bigl(
\sum_{j=0}^{s-1} h_{q,j} \gamma^{-j}
\Bigr).
\end{equation}
Then take $q=d$.
The lhs of \eqref{AprimeSPreste} is equal to 
$0$.
Thus 
$$P(\gamma^{-1}) =
 \gamma^{-d-1}
\Bigl(
\sum_{j=0}^{s-1} h_{d,j} \gamma^{-j}
\Bigr)
\qquad
\Longrightarrow
\qquad
P(\gamma) =
\sum_{j=0}^{s-1} h_{d,j} \gamma^{-j-1}.
$$
The height
of the polynomial
$W(X) :=\sum_{j=0}^{s-1} h_{d,j} X^{j+1}$ is $\leq$
 $(2^d -1) H + 2^d$.
We now assume
$|P(\gamma)| < \eta$.
By Kala-Vavra's Theorem \ref{kalavavra}
there exist an alphabet
$\mathcal{A} \subset \zb$,
a preperiod $R(X) \in \mathcal{A}[X]$, 
a period
$T(X) \in \mathcal{A}[X]$ such that
$$W(\gamma^{-1})=P(\gamma) = R(\gamma^{-1})
+ \gamma^{-\deg R -1} 
\sum_{j=0}^{\infty} 
\frac{1}{\gamma^{j (\deg T + 1)}}
T(\gamma^{-1})
$$
Since the relation
$S_{\gamma}^{*}(\gamma^{-1})
=
1 -t_{1} \gamma^{-1} -t_2 \gamma^{-2} - \ldots - t_{s-1} \gamma^{-s+1} 
-t_s \gamma^{-s} = 0$ holds, we may assume
$\deg R \leq s-1$,
$\deg T \leq s-1$.
Then, for $X$ specialized at $\gamma^{-1}$,
we have the identity
\begin{equation}
\label{Wrepresentation}
W(X) = R(X) +
X^{L}\frac{T(X)}{1-X^r}
\end{equation}
for some positive integers $L, r$.
The height of $(1-X^r) W(X)$ is 
$\leq 2 ((2^d -1) H + 2^d)$
and, with 
$\mathcal{A}$
assumed $=\{-m, \ldots,0,\ldots,+m\}$,
the height of 
$(1-X^r) R(X) + X^L T(X)$ is less than
$3 m$. Therefore $m$ is 
$\leq 2 ((2^d -1) H + 2^d)/3$.
We can take 
$m = \lceil 2((2^d -1) H + 2^d)/3 \rceil$.

Since the algebraic norm
N$(\gamma)$ is equal to
$\pm 1$ we cannot expect the uniqueness
of the representations \eqref{Wrepresentation},
for $X=\gamma^{-1}$, by
\cite{kovacskornyei}.
However, for any 
$(\gamma, \mathcal{A})$-eventually periodic
representation
of $P(\gamma)$
$$P(\gamma)= W(\gamma^{-1})
=
\frac{a'_{w}}{\gamma^{w}}
+\frac{a'_{w+1}}{\gamma^{w+1}}
+\frac{a'_{w+2}}{\gamma^{w+2}}
+\ldots,\qquad {\rm with}~
|a'_{j}| \leq m, j=w, w+1, \ldots$$
with $\,a'_{w} \neq 0$,
the exponent $w$ appearing
in the first term  tends to infinity if
$\eta$ tends to $0$.
Indeed, from Theorem 4, Remarks 5 to 7,
in \cite{bakermasakovapelantovavavra},
there exists a positive real
number $\kappa_{\gamma, \mathcal{A}} > 0$
such that $w$ is the minimal integer 
such that
$$\gamma^{w-1} \geq \frac{\kappa_{\gamma, \mathcal{A}}}
{|P(\gamma)|} \geq \frac{\kappa_{\gamma, \mathcal{A}}}
{\eta}.$$

\subsection{Application to Galois conjugation: convergence and eventually periodic representations along a sequence of almost Newman polynomials}
\label{S3.2}

Newman polynomials are polynomials with coefficients
in $\{0,1\}$. In \cite{dutykhvergergaugry}
almost Newman polynomials have been introduced:
an almost Newman polynomial is an integer polynomial
which has coefficients
in $\{0,1\}$ except the constant term equal to $-1$.

\begin{definition}
\label{almostNewmandefinition}
The collection of lacunary almost Newman  
polynomials
of the type:
$$f(x) := -1 + x + x^n +
x^{m_1} + x^{m_2} + \ldots + x^{m_s}$$
where $n \geq 2$, 
$s \geq 0$, 
$m_1 - n \geq n-1$, 
$m_{q+1}-m_q \geq n-1$ for $1 \leq q < s$,
is called
the {\it class
$\mathcal{B}$}. 
The case ``$s=0$" corresponds to 
the trinomials $G_{n}(z) :=
-1+z+z^n$. 
The subclass $\mathcal{B}_n
\subset \mathcal{B}$ is the
set of
polynomials
$f(x) \in \mathcal{B}$
whose third monomial is exactly
$x^n$, so that the union
$\mathcal{B} =
\cup_{n \geq 2} \mathcal{B}_n $
is disjoint.
\end{definition}

The ``Asymptotic Reducibility
Conjecture", formulated in 
\cite{dutykhvergergaugry}, says
that
$75\%$ of the polynomials $f(x) \in \mathcal{B}$ 
are irreducible. 
The factorization and the zeroes of
the polynomials of the class
$\mathcal{B}_n$, $n \geq 2$, have been
studied in \cite{dutykhvergergaugry}.

\begin{theorem}[Selmer \cite{selmer}]
Let $n \geq 2$. The trinomials $G_{n}(x)$ are irreducible
if $n \not\equiv 5 ~({\rm mod}~ 6)$, and, for
$n \equiv 5 ~({\rm mod}~ 6)$, are reducible as 
product of two irreducible factors whose
one is the cyclotomic factor $x^2 -x +1$,
the other factor 
$(-1 + x + x^n)/(x^2 - x +1)$
being nonreciprocal of degree
$n-2$.
\end{theorem}
 
By definition, for $n \geq 2$,
$\theta_{n} $ is 
the unique root
of the trinomial $-1+x+x^n$ in the interval
$(0,1)$. 
The algebraic integers
$\theta_{n}^{-1} > 1$ are Perron numbers. 
The sequence $(\theta_{n}^{-1})_{n \geq 2}$ 
is decreasing, tends to 1
if $n$ tends to $+\infty$.

\begin{theorem}[Dutykh - Verger-Gaugry \cite{dutykhvergergaugry}]
\label{thm1factorization}
For any $f \in \mathcal{B}_n$, $n \geq 3$,
denote by
$$f(x) = A(x) B(x) C(x) =
-1 + x + x^n +
x^{m_1} + x^{m_2} + \ldots + x^{m_s},$$
where $s \geq 1$, $m_1 - n \geq n-1$, 
$m_{j+1}-m_j \geq n-1$ for $1 \leq j < s$,
the factorization of  $f$ where $A$ is 
the cyclotomic part, 
$B$ the reciprocal noncyclotomic part,
$C$ the nonreciprocal part.
Then 
\begin{itemize}
\item[(i)]
the  nonreciprocal part $C$ is
nontrivial, irreducible 
and never vanishes on the unit circle,
\item[(ii)] if $\gamma > 1$
denotes the real algebraic number
uniquely determined 
by the sequence
$(n, m_1 , m_2 , \ldots , m_s)$ such that
$1/\gamma$ is the unique
real root of $f$ in
$(\theta_{n-1} , \theta_{n})$,
$-C^*(X)$ is the minimal polynomial $P_{\gamma}(X)$
of $\gamma$,
and $\gamma$ is a nonreciprocal algebraic integer.
\end{itemize}
\end{theorem}

\

Now let us assume the 
existence of a reciprocal  
algebraic integer
$\beta$ in the interval
$(\theta_{n}^{-1} , 
\theta_{n-1}^{-1})$
for some integer $n \geq 3$ ($n$ is fixed),
with ${\rm M}(\beta) < 1.176280\ldots$
Lehmer's number. 
It is canonically associated with, and
characterized by,
two analytic functions:
\begin{enumerate}
\item[(i)]
its minimal polynomial, say
$P_{\beta}$, which is monic and
reciprocal meaning\\
$X^{\deg P_{\beta}} P_{\beta}(1/X)= P_{\beta}(X)$;
denote $d := \deg P_{\beta}$,
$H:=$ the height of $P_{\beta}$;
the minimal polynomial
$P_{\beta}(X)$ of 
$\beta > 1$ can be written
\begin{equation}
\label{requalityPP}
P_{\beta}(X) = \widetilde{P_{\beta}}(X^r)
\end{equation}
for some integer $r \geq 1$ and
some $\zb$-minimal integer polynomial  
$\widetilde{P_{\beta}}(X)$.
The integer $r$ is the largest one
such that \eqref{requalityPP} holds; 
it depends upon $\beta$.
The $\beta$s such that $r \geq 2$ are excluded 
in the following.
\item[(ii)] the Parry Upper function 
$f_{\beta}(x)$ at
$\beta^{-1}$, which is the
generalized Fredholm
determinant of the $\beta$-transformation $T_{\beta}$
\cite{baladikeller}
which is a power series
with coefficients in the alphabet $\{0,1\}$
except the constant term equal to $-1$,
with distanciation between the exponents of the monomials:
$$f_{\beta}(x) := -1 + x + x^n +
x^{m_1} + x^{m_2} + \ldots + x^{m_s} +
\ldots$$
where $m_1 - n \geq n-1$, 
$m_{q+1}-m_q \geq n-1$ for $q \geq 1$.
$\beta^{-1}$ is the unique 
zero of $f_{\beta}(x)$
in the unit interval $(0,1)$.
The analytic function
$f_{\beta}(z)$ is related to the dynamical
zeta function $\zeta_{\beta}(z)$ of the
$\beta$-shift \cite{frougny} 
\cite{lagarias}
\cite{lothaire} by:
$f_{\beta}(z) = -1/\zeta_{\beta}(z)$.
Since $\beta$ is 
reciprocal, with the two real roots
$\beta$ and $1/\beta$, the series
$f_{\beta}(x)$ is never a polynomial,
by Descartes's rule on sign changes on the 
coefficient vector.
The algebraic integer $\beta$ is 
associated with the infinite
sequence of exponents $(m_j)$.  
\end{enumerate}

All the polynomial sections 
$S_{\gamma_{s}}^{*}(x)
:=-1 + x + x^n +
x^{m_1} + x^{m_2} + \ldots + x^{m_s}$
of $f_{\beta}(x)$
\, are polynomials of the class
$\mathcal{B}_n$, mostly
irreducible by the asymptotic reducibility conjecture, 
but not necessarily irreducible.
For every $s \geq 1$, 
denote by $\gamma_s > 1$ the (non-reciprocal)
algebraic integer which is such that
$\gamma_{s}^{-1}$ is the unique zero
in $(0,1)$
of the polynomial section
$S_{\gamma_{s}}^{*}(x)$
of $f_{\beta}(x)$.
We have: $\deg \gamma_{s}^{-1} =
\deg S_{\gamma_{s}}^{*}$ if and only if
$S_{\gamma_{s}}^{*}(x)$ is
irreducible.
Moreover $f_{\beta}(\beta^{-1})=0$ and
$\lim_{s \to \infty} \gamma_{s}
=
\beta$.

\

{\it We now apply Theorem \ref{mainthm}}:

\

The integer $n \geq 3$ is fixed.
For all $s$ such that
$\deg S_{\gamma_{s}}^{*}
\geq \deg P_{\beta}$, the identity
$$\mathbb{Q}(\gamma_s) =
{\rm Per}_{\mathcal{A}}(\gamma_s),$$
holds
with
$\mathcal{A}
=
\{-m, \ldots, +m\} \subset \zb$\,,
$m=\lceil2((2^{d} -1) H + 2^d)/3\rceil$.
By Theorem \ref{thm1factorization},
for any $s \geq 0$,
$\gamma_{s}^{-1}$ has no conjugate
on the unit circle.
The polynomial value $P_{\beta}(\gamma_s)
\in \mathbb{Q}(\gamma_s)$
is eventually periodic
\begin{equation}
\label{PeventPeriodic}
P_{\beta}(\gamma_s) 
=
R(\gamma_{s}^{-1}) + \frac{1}{\gamma_{s}^L} 
\sum_{j=0}^{\infty} \frac{1}{\gamma_{s}^{j \rho }} 
T(\gamma_{s}^{-1})
\qquad \in {\rm Per}_{\mathcal{A}}(\gamma_s)
\end{equation}
with $L, \rho$ and $R(X), T(X)$, depending upon
$s$. This representation of
$P_{\beta}(\gamma_s)$ 
starts as
$$
= \frac{a_{w, (s)}}{\gamma_{s}^w}
+\frac{a_{{w+1, (s)}}}{\gamma_{s}^{w+1}}
+\ldots
+\frac{a_{{w+m_s-1, (s)}}}{\gamma_{s}^{w+m_s-1}}
+ \ldots,
\quad a_{j,(s)} \in \mathcal{A}, 
j=w,  \ldots, w+m_s-1,$$
with $a_{w, (s)} \neq 0$ and
$w = w_s \geq 1$, depending upon
$s$,
satisfies
$\frac{\kappa_{\gamma_s , \mathcal{A}}}{\eta} 
\leq \gamma_{s}^{w_s -1}$ 
for some positive constant
$\kappa_{\gamma_s , \mathcal{A}}$ 
depending upon
$\gamma$ and $\mathcal{A}$.
Since $\mathcal{A}$ is independent of $s$,
and that the sequence 
$(\gamma_{s})$
is convergent with limit $\beta > 1$,
there exists a (true) constant $\kappa > 0$
such that 
$\frac{\kappa}{\eta} 
\leq \gamma_{s}^{w_s -1}$
from Theorem 4, Remarks 5 to 7,
in 
\cite{bakermasakovapelantovavavra}.
Since $\lim_{s \to \infty}
P_{\beta}(\gamma_s) =
0=
P_{\beta}(\beta)$, we take
$\eta = \eta_s :=|P_{\beta}(\gamma_s)|$.
The sequence $(\eta_s)$ tends to 0.
We deduce 
$\lim_{s \to \infty} w_s = +\infty$.

Let 
$\Omega \neq
\beta^{-1}$ be a zero
of modulus $< 1$
of $f_{\beta}(x)$.
We assume the existence
of a small disk
$D(\Omega, r)$ centered at
$\Omega$ of radius $r > 0$,
included in the open unit disk, which 
has the property that the only
zero of $f_{\beta}(x)$
in  $D(\Omega, r)$ is $\Omega$.
It is possible since the domain of 
existence of $f_{\beta}(x)$ is at least
the open unit disk
$D(0,1)$. 

The zero $\Omega$ is
limit point of a sequence of zeroes of
the polynomial sections of $f_{\beta}(x)$.
As soon as $s \geq s_0$
for some $s_0$, we assume that the disk 
$D(\Omega, r)$ contains 
only one zero of
$S_{\gamma_{s}}^{*}(x)$.
Denote by $r_s$ this zero. 
Let us assume that $r_s$ is a Galois conjugate
of $\gamma_{s}^{-1}$, and denote by $\sigma_s :
\gamma_s^{-1} \to r_s$ the 
$\qb$-automorphism of conjugation.
This assumption is reasonable by the 
Asymptotic Reducibility Conjecture 
which says that 75 \%
of the polynomial sections are irreducible.

The lenticular zeroes of
$f_{\beta}$ are peculiar zeroes, off the unit circle.
Let us briefly recall what is a lenticular zero
of $f_{\beta}$. Many examples of
lenticular zeroes are given in
\cite{dutykhvergergaugry}. 
The following theorem is Theorem 4 in
 \cite{dutykhvergergaugry}. 

\begin{theorem}
\label{thm2lenticuli}
Let $n \geq 260$.
There exist
two positive constants $c_n$ and $c_{A,n}$ , 
$c_{A,n} < c_n$,
such that the
roots of any
$f \in \mathcal{B}_n$,
$$f(x)
-1 + x + x^n +
x^{m_1} + x^{m_2} + \ldots + x^{m_s},$$
where $s \geq 1$, $m_1 - n \geq n-1$, 
$m_{j+1}-m_j \geq n-1$ for $1 \leq j < s$,
lying in
$- \pi/18 < \arg z < + \pi/18$
either belong to
$$\{z \in \cb \mid ||z|-1| < \frac{c_{A,n}}{n}\},\quad
\mbox{or to}\quad
\{z \in \cb \mid ||z|-1| \geq \frac{c_n}{n}\}.$$
\end{theorem}
The {\em lenticulus of zeroes $\omega$ 
of $f$} is then defined
as
$$\mathcal{L}_{\beta} :=
\{\omega \in \cb \mid f(\omega) =0, |\omega| < 1,
-\frac{\pi}{18} < \arg \omega < +\frac{\pi}{18},~
||\omega|-1| \geq \frac{c_n}{n}\}$$
where $1/\beta \in \mathcal{L}_{\beta}$ 
is the positive real
zero of $f$. If a zero of 
$f$ belongs to $\mathcal{L}_{\beta}$ 
we say that it is a 
{\em lenticular zero of $f$}.

Let us go back to the above assumption.
If $\Omega$ is a lenticular
zero of $f_{\beta}(x)$, then, by
\cite{dutykhvergergaugry}, all the polynomial
sections $S_{\gamma_{s}}^{*}(x)$ do have
also a (unique) lenticular zero close to $\Omega$
which is a conjugate of $\gamma_{s}^{-1}$.
For the non-lenticular zeroes of $f_{\beta}(x)$,
very close to the unit circle,
the above assumption is necessary.

To summarize, for $s \geq s_0$:
$$f_{\beta}(\Omega)=0,
\quad S_{\gamma_{s}}^{*}(r_s)  =0, \quad
r_s = \sigma_{s}(\gamma_s^{-1}),
\quad
|\Omega -\sigma_{s}(\gamma_s^{-1})|<r,
\quad
\lim_{s \to \infty} r_s = \Omega.
$$
Let us show that 
$P_{\beta}(\Omega)=0$.
Let us conjugate \eqref{Wrepresentation}
for $X=\gamma_{s}^{-1}$.
The power series 
\eqref{PeventPeriodic}
specialized at $\gamma_{s}^{-1}$
is eventually periodic, therefore can be conjugated
term by term, once the image of
$\gamma_{s}^{-1}$
by the conjugation $\sigma_s$ is such that
$|\sigma_{s}(\gamma_{s}^{-1})|<1$, to ensure
convergence.
Then
\begin{equation}
\label{Wrepresentation_conjug}
\sigma_{s}(P_{\beta}(\gamma_s))=
W(r_s) = R(r_s) +
r_{s}^{L}\frac{T(r_s)}{1-r_{s}^{\rho}}
=
a_{w_{s}, (s)} r_{s}^{w_s}
+a_{{w_{s}+1, (s)}} r_{s}^{w_{s}+1}
+\ldots
\end{equation}
with 
$$\frac{\kappa}{|P_{\beta}(\gamma_s)|} 
\leq \gamma_{s}^{w_s -1} \qquad
{\rm and}
\qquad
 |r_s|<|\Omega|+r < 1,
 \qquad s \geq s_0.$$
We have, with 
$m = \lceil2((2^{d} -1) H + 2^d)/3\rceil$,
$$|W(r_s)|
\leq
|a_{w_{s}, (s)}| |r_{s}^{w_s}|
+ |a_{{w_{s}+1, (s)}}| |r_{s}^{w_{s}+1}|
+\ldots
$$
$$\leq m \bigl(
|r_{s}^{w_s}| + |r_{s}^{w_{s}+1}|+\ldots
\bigr)
=m 
|r_{s}^{w_s}| \bigl(
1 + |r_s| +|r_s|^2 +\ldots
\bigr).
$$
Then
\begin{equation}
\label{tagga}
|W(r_s)| \leq |r_s|^{w_s} \frac{m}{1-|r_s|}.
\end{equation}
Let us observe that within a period of period length
$\rho$ in 
the power series
\eqref{Wrepresentation_conjug} 
a certain number of coefficients are
equal to zero, and therefore that the upper bound
\eqref{tagga} can be improved using the period 
length $\rho$ and the degree of $T$. 
However it is sufficient
for below.
We deduce 
$$P_{\beta}(\Omega) = \lim_{s \to \infty} W(r_s) =0.$$
Under the above assumptions, we have proved:

\begin{proposition}
\label{omegazerofbetaPbeta}
$$f_{\beta}(\Omega) = 0 \qquad \Longrightarrow
\qquad
P_{\beta}(\Omega)=0.
$$
\end{proposition}
\begin{remark}
As a consequence the properties of the 
analytic function
$$f_{\beta}(x) := -1 + x + x^n +
x^{m_1} + x^{m_2} + \ldots + x^{m_s} +
\ldots,\qquad
|x| < 1,$$
where $m_1 - n \geq n-1$, 
$m_{q+1}-m_q \geq n-1$ for $q \geq 1$
can be used to investigate the geometry of the
zeroes of the polynomial $P_{\beta}(X)$,
in particular  the existence of 
integer polynomials having
a very small Mahler measure below Lehmer's number
$1.176280\ldots$. 
Let us note that
the Parry Upper function $f_{\beta}(z)$
at $\beta$ (here reciprocal) is related to the dynamical
zeta function $\zeta_{\beta}(z)$ of the 
$\beta$-shift by:
$f_{\beta}(z) = -1/\zeta_{\beta}(z)$
(\cite{lagarias};
see Solomyak \cite{solomyak} for the zeroes). 
In this respect, Proposition
\ref{omegazerofbetaPbeta} says that the poles
of modulus $< 1$ of the meromorphic extension
of $\zeta_{\beta}(z)$ in the open unit disk are,
under some assumptions (as mentioned above), zeroes
of the minimal polynomial of
$\beta$. 
\end{remark}

\subsection{Natural and intermediate alphabets along sequences of almost Newman polynomials: examples}
\label{S3.3}

When the base $\gamma$
is fixed, as in Theorem \ref{mainthm},
the intermediate alphabet produced by 
a rewriting trail
has a size growing linearly with 
the height $H$ of the polynomial
$P$.
Leaving $\gamma$ fixed and varying $H$
in $P$, 
when $H$ becomes very large, this 
intermediate alphabet reaches the natural alphabet
$\mathcal{A}_{\gamma}^{(nat)}$, 
becomes equal to it, exceeds it;
so that there is no interest to
proceed with rewriting trails for such polynomials
having a large height, i.e. for $H$
above a certain critical value.

The natural alphabets 
$\mathcal{A}_{\gamma_s}^{(nat)}$
along the sequence
of the polynomial sections
$S_{\gamma_j}^{*}(X)
=-1+x+x^n + x^{m_1} + x^{m_2} + \ldots 
+ x^{m_j}$,  $j \geq 1$,
of a given Parry Upper function
\begin{equation}
\label{fbeta}
f_{\beta}(z) =
-1+x+x^n + x^{m_1} + x^{m_2} + \ldots 
+ x^{m_j}+\ldots
\end{equation}
i.e.
along a sequence of bases
$(\gamma_s)$,
as in Section \ref{S3.2}, 
are huge and
do not remain constant.
Let us take examples.
The following 
alphabets $\mathcal{A}_{\gamma_j}^{(nat)}
=\{-m_{\gamma_j}, \ldots, +m_{\gamma_j}\}$ are calculated
by \eqref{Alphabet} and
Proposition \ref{poly_dominant}.
Denote $m= m_{\gamma_j}$ for short.
The integer $j$ is the number of monomials 
added to the trinomial $-1+x+x^{11}$. All 
the polynomials are irreducible, 
of the same degree
($= 101$), and belong to the class
$\mathcal{B}_{11}$:
\vspace{0.15cm}

\parindent=0cm
j =1:  p := $x^{101} + x^{11} + x - 1$,\\   
$m =2.617526038*10^{365}$
\vspace{0.15cm}

j = 2: p := $x^{101} + x^{21} + x^{11} + x - 1$,\\ 
$m = 4.088496786*10^{288}$;
\vspace{0.15cm}

j = 3: p := $x^{101} + x^{35} + x^{21} + x^{11} + x - 1$, \\
$m  = 3.196582086*10^{151}$;
\vspace{0.15cm}

j = 4: p := $x^{101} + x^{45} + x^{35} + x^{21} + x^{11} + x - 1$, \\
$m = 3.823048784*10^{462}$;
\vspace{0.15cm}

j = 5: p := $x^{101} + x^{57} + x^{45} + x^{35} + x^{21} + x^{11} + x - 1$, \\
$m = 8.866692051*10^{248}$;
\vspace{0.15cm}

j = 6: p := $x^{101} + x^{69} + x^{57} + x^{45} + x^{35} + x^{21} + x^{11} + x - 1$, \\
$m = 4.851172757*10^{224}$;
\vspace{0.15cm}

j = 7: p := $x^{101} + x^{80} + x^{69} + x^{57} + x^{45} + x^{35} + x^{21} + x^{11} + x - 1$,\\ 
$m = 6.062823380*10^{222}$;
\vspace{0.15cm}

j = 7: p := $x^{101} + x^{81} + x^{69} + x^{57} + x^{45} + x^{35} + x^{21} + x^{11} + x - 1$, \\
$m = 4.617819094*10^{1083}$;
\vspace{0.15cm}

j = 8: p := $x^{101} + x^{91} + x^{80} + x^{69} + x^{57} + x^{45} + x^{35} + x^{21} + x^{11} + x - 1$,\\ 
$m = 2.085371358*10^{536}$;
\vspace{0.15cm}

j = 8: p := $x^{101} + x^{90} + x^{80} + x^{69} + x^{57} + x^{45} + x^{35} + x^{21} + x^{11} + x - 1$, \\
$m = 3.484819567*10^{196}$;
\vspace{0.15cm}

No simple law of $m_{\gamma_{j}}$
appears as a function of $j$:
for $j=7$, resp. $j=8$,
a big difference appear in the size of the alphabets,
obtained by varying just one
monomial in the definition of p.
 
\parindent=0.4cm

On the contrary, the alphabets obtained 
by rewriting trails
along the sequence
of the polynomial sections
$S_{\gamma_j}^{*}(X)$,
from a given polynomial, remain constant.

Given $f_{\beta}(x)$ as
in \eqref{fbeta}, the growth rate of the natural
alphabets $\mathcal{A}_{\gamma_j}^{(nat)}$
with the degree 
$m_j$ and the number $j$ of monomials
is investigated, in the Figures
\ref{mahler_mj77} to \ref{mahler_periodic},
in terms of the Mahler measures 
$${\rm M}(S_{\gamma_j}^{*})
=
{\rm M}(-1+x+x^n + x^{m_1} + x^{m_2} + \ldots 
+ x^{m_j});$$
this is sufficient according
to the approximate lower bounds
$$\frac{1}{2}{\rm M}(\gamma_j)^{N_j -1}
|S_{\gamma_j}^{*}(1)|~=~
\frac{j+1}{2}\,{\rm M}(\gamma_j)^{N_j -1}
$$
by \eqref{002conditionMahler},
where the $j$-th integer $N_j$ is the
dominance index relative to
$\gamma_j$,  and
that the identities
${\rm M}(S_{\gamma_j}^{*}) =
{\rm M}(\gamma_j)$,
$S_{\gamma_j}(X)= P_{\gamma_j}(X)$,
hold for 75 \%
of the polynomial sections,
by the Asymptotic
Reducibility Conjecture.
Mahler measures are calculated
by means of Graeffe's method 
\cite{flammangrhinsacepee}
and in PARI/GP \cite{parigp}.
The Mahler measure
${\rm M}(S_{\gamma_j}^{*})$
is a function of the geometry
of the roots of $S_{\gamma_j}^{*}(X)$
which lie inside the open unit disk;
the respective 
roles of the non-lenticular roots
with respect to the lenticular roots
\cite{dutykhvergergaugry}
are investigated in 
 \cite{dutykhvergergaugry2}.
The fairly large values $N_j$ of the dominance indices,
arising from the arithmetics of the
iteration of the companion matrix of
$P_{\beta}$, are not indicated.
The values $N_j$ are related to the dynamical
system with polynomial action, see 
Remark \ref{remarkSysDyn02}.

\begin{figure}
\begin{center}
\includegraphics[width=5.5cm]{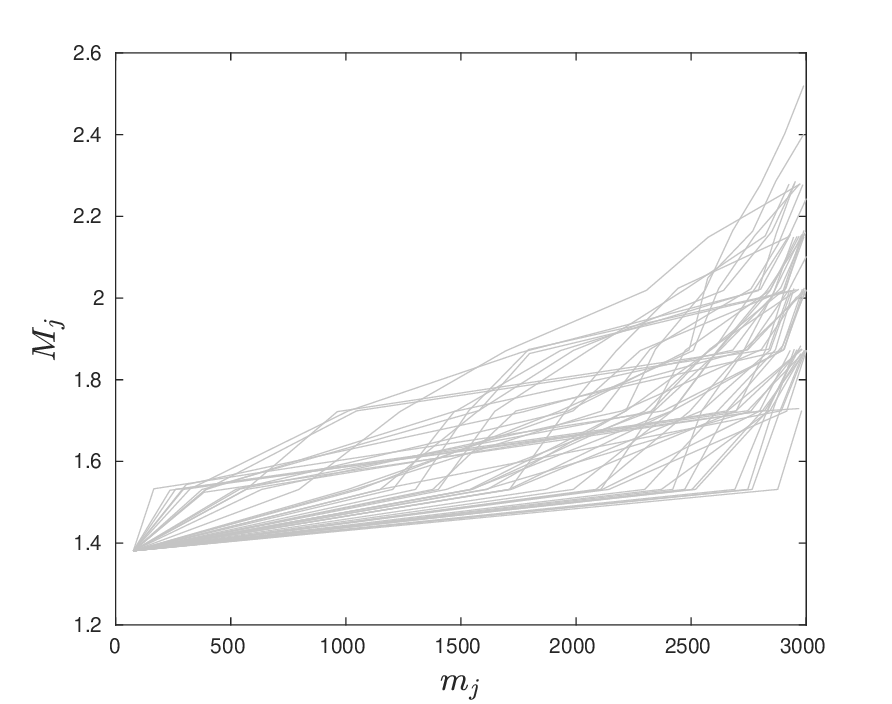}
\end{center}
\caption{Mahler measures ${\rm M}_j
= {\rm M}(S_{\gamma_j}^{*})$
of $j$th-polynomial 
sections $S_{\gamma_j}^{*}$ of
$f_{\beta}(z)$ for
various $\theta_{77}^{-1} < \beta
< \theta_{76}^{-1}$,
as a function of the degree 
$m_j := \deg S_{\gamma_j}^{*}$.
The initial value is ${\rm M}_0
= {\rm M}(-1+x+x^{77}) \approx 1.38$ 
by \cite{flammang} \cite{vergergaugry2}.}
\label{mahler_mj77}
\end{figure}

\begin{figure}
\begin{center}
\includegraphics[width=5.5cm]{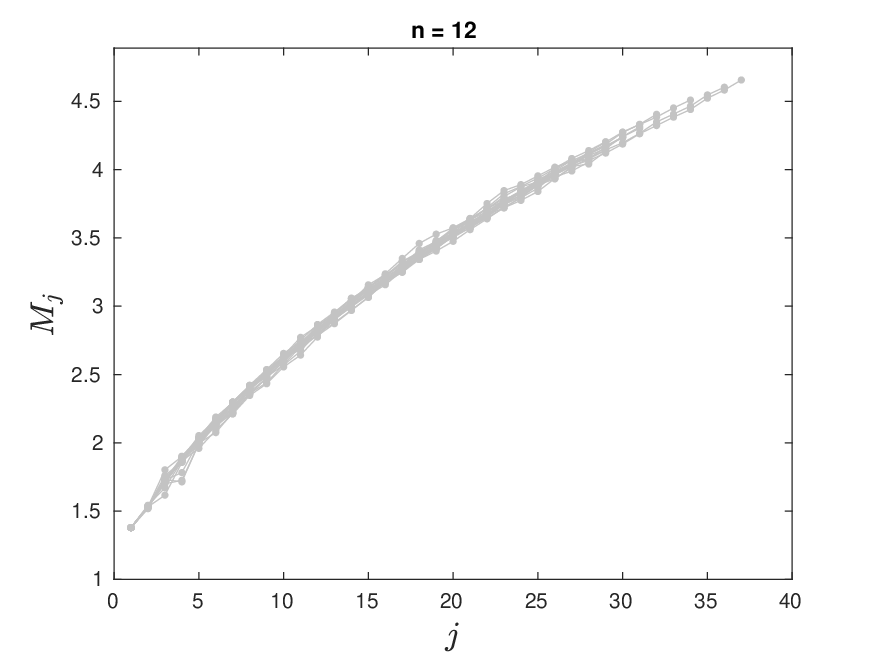} 
\end{center}
\caption{
Mahler measures ${\rm M}_j
= {\rm M}(S_{\gamma_j}^{*})$
of $j$th-polynomial 
sections $S_{\gamma_j}^{*}$ of
$f_{\beta}(z)$ for
various $\theta_{12}^{-1} < \beta
< \theta_{11}^{-1}$,
as a function of the number
$j$ of monomials added to
$-1+x+x^{12}$.
The initial value is 
${\rm M}(-1+x+x^{12}) \approx 1.38$ 
by \cite{flammang} \cite{vergergaugry2}.}
\label{mahler_12}
\end{figure}

\begin{figure}
\begin{center}
\includegraphics[width=5.5cm]{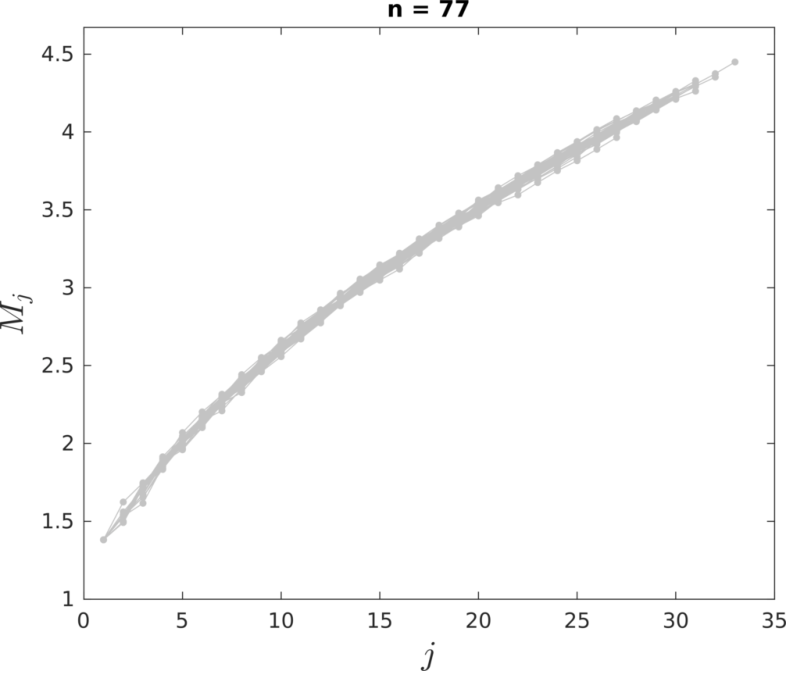} 
\end{center}
\caption{
Mahler measures ${\rm M}_j
= {\rm M}(S_{\gamma_j}^{*})$
of $j$th-polynomial 
sections $S_{\gamma_j}^{*}$ of
$f_{\beta}(z)$ for
various $\theta_{77}^{-1} < \beta
< \theta_{76}^{-1}$,
as a function of the number
$j$ of monomials added to
$-1+x+x^{77}$.
The initial value is 
${\rm M}(-1+x+x^{77}) \approx 1.38$ 
by \cite{flammang} \cite{vergergaugry2}.}
\label{mahler_77}
\end{figure}

\begin{figure}
\begin{center}
\includegraphics[width=8cm]{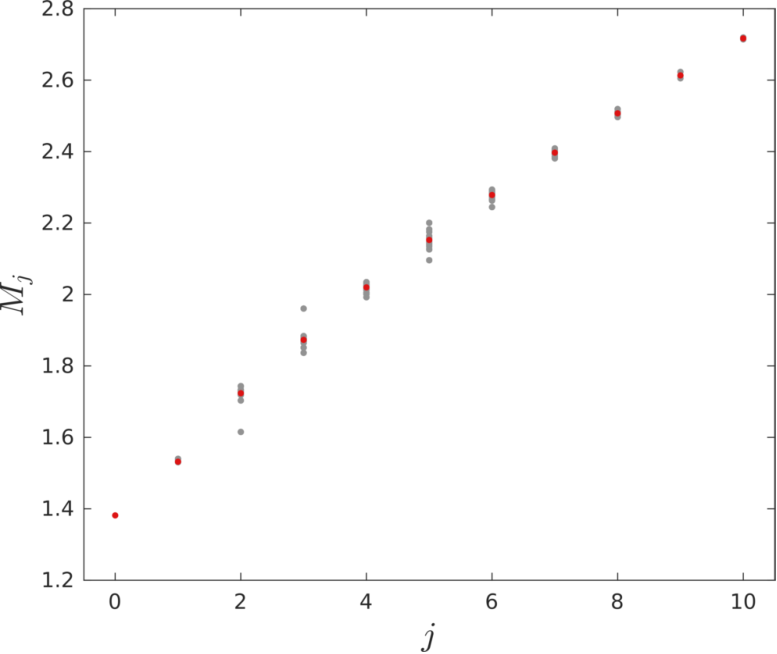} 
\end{center}
\caption{
Mahler measures ${\rm M}_j
= {\rm M}(S_{\gamma_j}^{*})$
of $j$th-polynomial 
sections $S_{\gamma_j}^{*}$ of
$f_{\beta}(z)$ for
various $\theta_{149}^{-1} < \beta
< \theta_{148}^{-1}$,
as a function of the number
$j$ of monomials added to
$-1+x+x^{149}$.
The initial value is 
${\rm M}(-1+x+x^{149}) \approx 1.38$ 
by \cite{flammang} \cite{vergergaugry2}.}
\label{mahler_149}
\end{figure}

\begin{figure}
\begin{center}
\includegraphics[width=8cm]{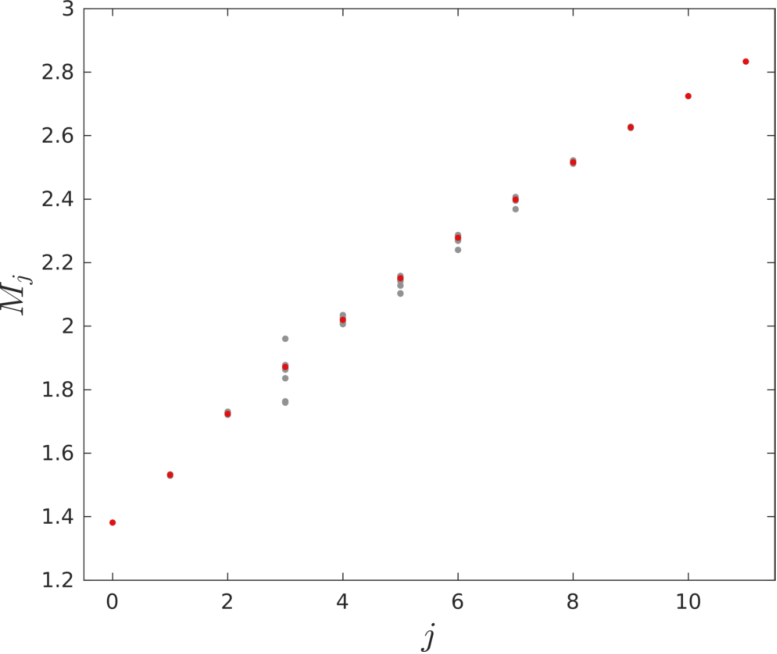} 
\end{center}
\caption{
Mahler measures ${\rm M}_j
= {\rm M}(S_{\gamma_j}^{*})$
of $j$th-polynomial 
sections $S_{\gamma_j}^{*}$ of
$f_{\beta}(z)$ for
various $\theta_{220}^{-1} < \beta
< \theta_{219}^{-1}$,
as a function of the number
$j$ of monomials added to
$-1+x+x^{220}$.
The initial value is 
${\rm M}(-1+x+x^{220}) \approx 1.38$ 
by \cite{flammang} \cite{vergergaugry2}.}
\label{mahler_220}
\end{figure}

\begin{figure}
\begin{center}
\includegraphics[width=6.5cm]{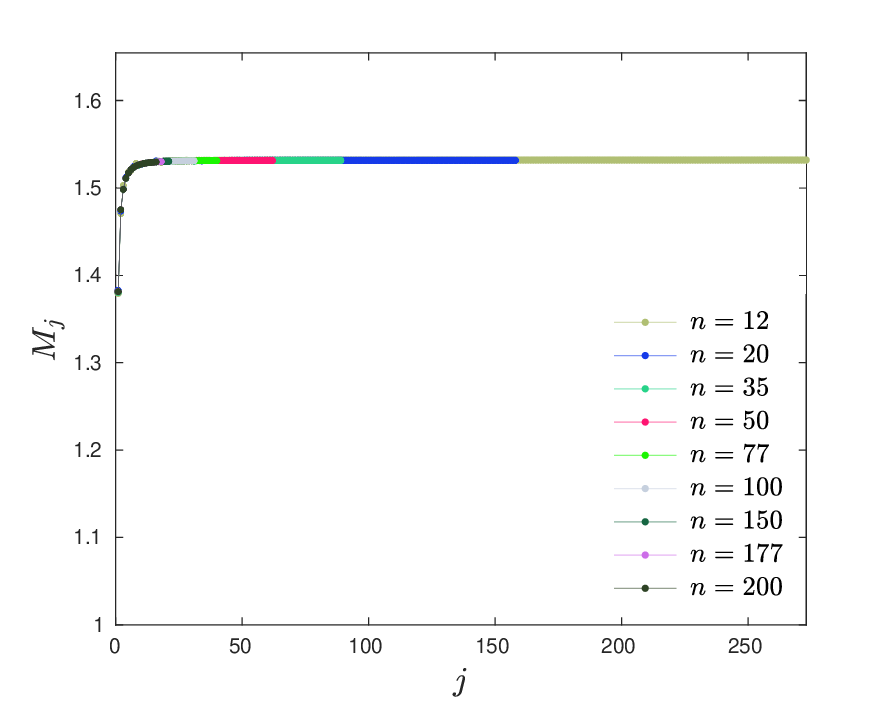} 
\end{center}
\caption{Minimal lacunarity:
Mahler measure ${\rm M}_j
= {\rm M}(S_{\gamma_j}^{*})$
of the eventually periodic
section $S_{\gamma_j}^{*}(x)
= -1 +x +x^n +x^{m_1} + \ldots 
+x^{m_j}$, $m_1 - n = n-1$,
$m_{q+1} - m_q = n-1$ for $q \geq 1$,
as a function of the number of monomials $j$
added to the trinomial
$-1+x+x^n$, for various values of $n$.
The initial value is 
${\rm M}(-1+x+x^{n}) \approx 1.38$ 
by \cite{flammang} \cite{vergergaugry2}.
The increase of ${\rm M}_j$ is 
followed by a plateau.}
\label{mahler_periodic}
\end{figure}

Along the sequence
of the polynomial sections
$S_{\gamma_j}^{*}(X)$
of $f_{\beta}(z)$, for 
$\theta_{n}^{-1} < \beta < \theta_{n-1}^{-1}$,
$n \geq 3$,
any algebraic integer, 
the sequence of the exponents
$(m_j)$ satisfies
\begin{equation}
\label{mqplusunmq}
1 + \frac{n-1}{m_j} 
\leq 
\frac{m_{j+1}}{m_j} ,
\qquad
\limsup_{j \to \infty} \frac{m_{j+1}}{m_j}
\leq 
\frac{\lo({\rm M}(\beta))}{\lo \beta}
 \end{equation}
by Theorem 1.1 in \cite{vergergaugry06}, so
that the lacunarity in
$f_{\beta}(z)$ remains moderate, and 
the number of monomials
in a section $S_{\gamma_j}^{*}(X)$ is always
increasing with $j$ with positive density. 
The topological entropies 
$\lo {\rm M} (\beta)$
and $\lo \beta$
are related to the two dynamical systems
involved in the sequences of coefficients
of $f_{\beta}(z)$,
see Remark \ref{remarkSysDyn02}
and \cite{vergergaugry06}.

In Figure \ref{mahler_12} to Figure
\ref{mahler_220} 
the Mahler measures ${\rm M}_j
= {\rm M}(S_{\gamma_j}^{*})$
of the $j$th-polynomial 
sections $S_{\gamma_j}^{*}$ of
$f_{\beta}(z)$ are represented for
various $\theta_{n}^{-1} < \beta
< \theta_{n-1}^{-1}$,
as a function of the number
$j$ of monomials added to
$-1+x+x^{n}$, for different values of $n$:
$n=12, 77, 149, 220$.
The initial value is 
${\rm M}(-1+x+x^{n}) \approx 1.38$ 
by 
\cite{flammang},
\cite{vergergaugry2}.
The growth rates are close to obey
a linear growth with $j$.
Each time, the growth of ${\rm M}_j$ 
occurs with $j$, without stabilization except
in Figure \ref{mahler_periodic} where a plateau appears
when the sequence of exponents 
$(m_j)$ is purely periodic (with period length
$n-1$).

\

\

\

\section*{Acknowledgements}

We would like to thank the anonymous referee for his helpful comments.


\frenchspacing

\begin{thebibliography}{99}

\bibitem
{bakermasakovapelantovavavra}
    \textsc{S. Baker, Z. Mas\'akov\'a, E. Pelantov\'a 
    {\rm and} T. V\'avra},
    {\it On Periodic Representations in non-Pisot Bases},
    Monatsh. Math. {\bf 184} (2017), 1--19. 

\bibitem
{baladikeller}
    \textsc{V. Baladi {\rm and} Keller},
    {\it Zeta Functions and Transfer Operators 
    for Piecewise Monotone Transformations},
    Comm. Math. Phys. {\bf 127} (1990), 459--477. 
    
\bibitem
{bertinetal}
    \textsc{M.-J. Bertin, A. Decomps-Guilloux, 
    M. Grandet-Hugot, M. Pathiaux-Delefosse {\rm and}
    J.P. Schreiber},
    {\it Pisot and Salem Numbers},
    Birkh\"auser Verlag (1992).

\bibitem
{daroczykatai}
    \textsc{Z. Dar\'oczy {\rm and} I. K\'atai},
    {\it Generalized Number Systems in the Complex Plane},
    Acta Math. Hungar. {\bf 51} (1988), 409--416.

\bibitem
{dutykhvergergaugry}
    \textsc{D. Dutykh {\rm and} J.-L. Verger-Gaugry},
    {\it On the Reductibility and the Lenticular
    Sets of Zeroes of Almost Newman Lacunary Polynomials},
    Arnold Math. J. {\bf 4} (2018), 315--344.

\bibitem
{dutykhvergergaugry2}
    \textsc{D. Dutykh {\rm and} J.-L. Verger-Gaugry},
    {\it Spirals of Poles of the Dynamical Zeta 
    Function of the $\beta$-shift for $\beta$ close to one,
    and Lehmer's Problem},
    preprint (2021). 

\bibitem
{einsiedlereverestward}
    \textsc{M. Einsiedler, G. Everest {\rm and} T. Ward},
    {\it Primes in Sequences Associated to Polynomials
    (after Lehmer)},
    LMS. J. Comput. Math. {\bf 3} (2000),
    125--139.

\bibitem
{everestward}
    \textsc{G. Everest {\rm and} T. Ward},
    {\it Heights of Polynomials and Entropy in 
    Algebraic Dynamics},
    Springer-Verlag, 1999.

\bibitem
{flammang}
    \textsc{V. Flammang},
    {\it The Mahler Measure of Trinomials of Height 1},
    J. Austral. Math. Soc {\bf 96} (2014), 231--243.

\bibitem
{flammangrhinsacepee}
    \textsc{V. Flammang, G. Rhin {\rm and} J.-M. 
    Sac-Ep\'ee},
    {\it Integer Transfinite Diameter and Polynomials 
    with Small Mahler Measure},
    Math. Comp. {\bf 75:255} (2006), 1527--1540.

\bibitem
{frougny}
    \textsc{C. Frougny},
    Chap. 7 ``Numeration Systems" in \cite{lothaire}. 
 
\bibitem
{frougnypelantovasvobodova} 
    \textsc{C. Frougny, E. Pelantov\'a {\rm and}
    M. Svobodov\'a},
    {\it Parallel Addition in Non-standard 
    Numeration Systems},
    Theor. Comput. Sci. {\bf 412} (2011), 5714--5727.
 
\bibitem
{kalavavra}
    \textsc{V. Kala {\rm and} T. V\'avra},
    {\it Periodic Representations in Algebraic Bases},
    Monatsh. Math. {\bf 188} (2019), 109--119.

\bibitem
{kestelman}
    \textsc{H. Kestelman},
    {\it Automorphisms of the field of complex numbers},
    Proc. London Math. Soc. {\bf 53} (1951), 1--12.

\bibitem
{kovacs}
    \textsc{B. Kov\'acs},
    {\it Canonical Number Systems in Algebraic 
    Number Fields},
    Acta Math. Hungar. {\bf 37} (1981), 405--407.
    
\bibitem
{kovacskornyei}
    \textsc{B. Kov\'acs {\rm and} I. 
    K\"ornyei},
    {\it On the Periodicity of the Radix Expansion},
    Ann. Univ. Sci. Budapest. Sect. Comput. {\bf 13} 
    (1992), 129--133. 

\bibitem
{lagarias}
    \textsc{J.C. Lagarias},
    {\it Number Theory Zeta Functions and Dynamical 
    Zeta Functions}, in
    Spectral problems in Geometry and Arithmetic,
    Contemp.Math., {\bf 237},
    Amer. Math. Soc., Providence, RI (1999), 45--86.

\bibitem
{lancaster}
    \textsc{P. Lancaster},
    {\it Theory of Matrices},
    Academic Press (1969).    
    
\bibitem
{lehmer}
    \textsc{D.H. Lehmer},
    {\it Factorization of Certain Cyclotomic Functions},
    Ann. of Math. {\bf 34} (1933), 461--479.    

\bibitem
{lind} 
    \textsc{D.A. Lind} ,
    {\it Dynamical Properties of Quasihyperbolic 
    Toral Automorphisms}, 
    Ergodic Theory Dynam. Systems {\bf 2} (1982), 49--68.

\bibitem
{lothaire}
    \textsc{M. Lothaire},
    {\it Algebraic Combinatorics on Words},
    Cambridge University Press, Cambridge (2002).
    
    
\bibitem
{pethoe}
    \textsc{A. Peth\H{o}},
    {\it On the Periodic Expansion of Algebraic Numbers},
    Ann. Univ. Sci. Budapest. Sect. Comput. {\bf 18} 
    (1999), 167--174. 
    
\bibitem
{schmidt}
    \textsc{K. Schmidt},
     {\it On Periodic Expansions of Pisot Numbers 
     and Salem Numbers},
     Bull. London Math. Soc. {\bf 12} (1980), 269--278.

\bibitem
{selmer}
    \textsc{E.S. Selmer},
    {\it On the Irreducibility of Certain Trinomials},
    Math. Scand. {\bf 4} (1956), 287--302.
    
\bibitem
{solomyak}
    \textsc{B. Solomyak},
    {\it Conjugates of beta-numbers and the zero-free 
    domain for a class of analytic functions},
    Proc. London Math. Soc. {\bf 68} (1994), 477--498.

\bibitem
{parigp}    
    \textsc{The PARI~Group},
    {\it PARI/GP version \texttt{2.11.2} (2019),
     ``Univ. Bordeaux"},
     available from {http://pari.math.u-bordeaux.fr/}
    
\bibitem
{thurston}
    \textsc{W.P. Thurston},
    {\it Groups, Tilings and Finite State Automata: 
    Summer 1989 AMS Colloquium Lectures},
    Research Report GCG, Geometry Computing Group, 1989.    
    
\bibitem
{vergergaugry06}
     \textsc{J.-L. Verger-Gaugry},
     {\it On Gaps in R\'enyi $\beta$-expansions of Unity
     for $\beta > 1$ an Algebraic Number},
     Annales Inst. Fourier Grenoble {\bf 56}
     (2006), 2565--2579.

\bibitem
{vergergaugry2}
     \textsc{J.-L. Verger-Gaugry},
     {\it On the Conjecture of Lehmer,  
     Limit Mahler Measure of Trinomials
     and Asymptotic Expansions},
     Uniform Distribution Theory J.      
     {\bf 11} (2016), 79--139.    



\bibitem
{yales}
    \textsc{P.B. Yales}
    {\it Automorphisms of the Complex Numbers},
    Math. Mag. {\bf 39} (1966), 135--141.
\end{thebibliography}
\end{document}